\documentclass{amsart}
\usepackage{amsmath,amssymb,graphicx,url,psfrag} 
\newtheorem{thm}{Theorem}
\newtheorem{lem}[thm]{Lemma}
\newtheorem{cor}[thm]{Corollary}
\newtheorem{prop}[thm]{Proposition}
\theoremstyle{definition}

\newcommand{\bi}{\begin{itemize}}
\newcommand{\ei}{\end{itemize}}
\newcommand{\be}{\begin{enumerate}}
\newcommand{\ee}{\end{enumerate}}
\newcommand{\bc}{\begin{center}}
\newcommand{\ec}{\end{center}}
\newcommand{\bt}{\begin{tabular}}
\newcommand{\et}{\end{tabular}}
\newcommand{\ba}{\begin{array}}
\newcommand{\ea}{\end{array}}

\newcommand{\Diek}{Diek}
\newcommand{\EEO}{EEO}

\newcommand{\Miller}{Miller}
\newcommand{\Ushakov}{Ushakov}
\newcommand{\Groves}{MR1425318}

\newcommand{\WordProc}{MR1161694}

\begin{document}


\title[A linear-time algorithm to compute geodesics]
{A linear-time algorithm to  compute geodesics in  solvable Baumslag-Solitar groups}

\author[M. Elder] {Murray Elder}
\address{Mathematics, University of Queensland, Brisbane, Australia}
\email[url]{murrayelder@gmail.com,
http://sites.google.com/site/melderau/}

 \date{\today}

\begin{abstract}
We present an algorithm to convert a word of length $n$ in the standard generators of the solvable Baumslag-Solitar group $BS(1,p)$ into a geodesic word, which runs in linear time and $O(n\log n)$ space 
on a random access machine.
\end{abstract}

\maketitle

\section{Introduction}

Recently Miasnikov, Roman'kov, Ushakov and Vershik \cite{\Ushakov} proved that for free metabelian groups with  standard generating sets, the following problem is NP-complete: \bi\item given a word in the generators and an integer $k$, decide whether the geodesic length of the word is less than $k$.\ei
They call this the {\em bounded geodesic length problem} for a group $G$ with finite generating set $\mathcal G$.
It follows that given a word, computing its geodesic length, and finding an explicit geodesic representative for it, are NP-hard problems. These problems are referred to as the {\em geodesic length problem} and the {\em geodesic problem} respectively.

In this article we consider the same problems for a different class of metabelian groups, the well known
{\em Baumslag-Solitar groups}, with presentations
\[\langle a, t \; | \; tat^{-1}=a^p\rangle\] for any integer $p\geq 2$.
We give a deterministic algorithm which takes as input a word   in the generators $a^{\pm 1}, t^{\pm 1}$ of length $n$, and outputs a geodesic word representing the same group element, 
in   time $O(n)$. Consequently, the three problems are solvable in linear time\footnote{It is clear that solving the geodesic problem implies the other two. In \cite{ER} the author and Rechnitzer show they are all in fact equivalent}.

In an unpublished preprint  \cite{\Miller} Miller gives a procedure to convert a word in the above group to a geodesic of the form $t^{-k}zt^{-m}$ where $z$ belongs to a regular language over the alphabet $\{a,a^{-1},t\}$. Miller's algorithm would take exponential time in the worst case. The algorithm presented here  follows Miller's procedure with some modifications (using pointers in part one and a careful tracking procedure in part two) to ensure linear time and $O(n\log n)$ space. Also, our algorithm does not output normal forms -- the geodesic output depends on the input word. A geodesic normal form is easily obtainable however, if one first runs a (polynomial time) algorithm to convert input words into a normal form (see for example \cite{\EEO}).

We use as our computational model a {\em random access machine}, which allows us to access (read, write and delete) any specified  location in an array in constant time.

Recent work of Diekert and Laun \cite{\Diek} extends the result of this paper to groups of the form $\langle a, t \; | \; ta^pt^{-1}=a^q\rangle$ when $p$ divides $q$. Their algorithm runs in quadratic time, but in the case $p=1$ the time reduces to linear, although their algorithm is qualitatively different.

The author wishes to thank Sasha Ushakov and Alexei Miasnikov for suggesting this problem, and Sasha, Alexei, Alex Myasnikov, Igor Lysionok, Andrew Rechnitzer and Yves Stalder for many helpful suggestions. The author thanks the  anonymous reviewer for their careful reading and instructive comments and corrections.
The author gratefully  acknowledges the support of the Algebraic Cryptography Center at Stevens Institute of Technology.

\section{Preliminaries}

Fix $G_p$ to be the Baumslag-Solitar group $\langle a,t \; |\; t^{-1}at=a^p\rangle$ for some $p\geq 2$.
We will call a single relator a  {\em brick}, with sides labeled by $t$ edges, as in Figure \ref{fig:brick-sheet}. The Cayley graph can be obtained by gluing together these bricks. We call a 
 {\em sheet} a subset of the Cayley graph made by laying rows of bricks atop each other to make a plane, also shown in Figure \ref{fig:brick-sheet}. 
\begin{figure}[ht!]
\bc
\bt{cc}
\psfrag{t}{$t$}\psfrag{a}{$a$} 
\includegraphics[scale=.34]{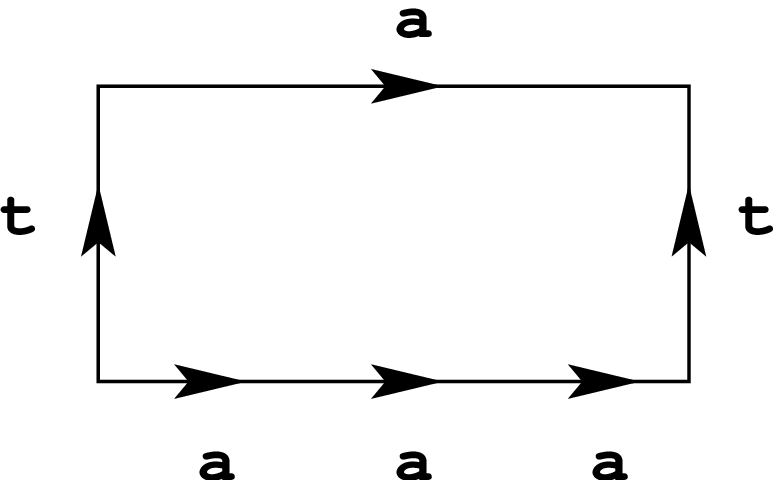} &
\psfrag{t}{$t$}\psfrag{a}{$a$} \includegraphics[scale=.34]{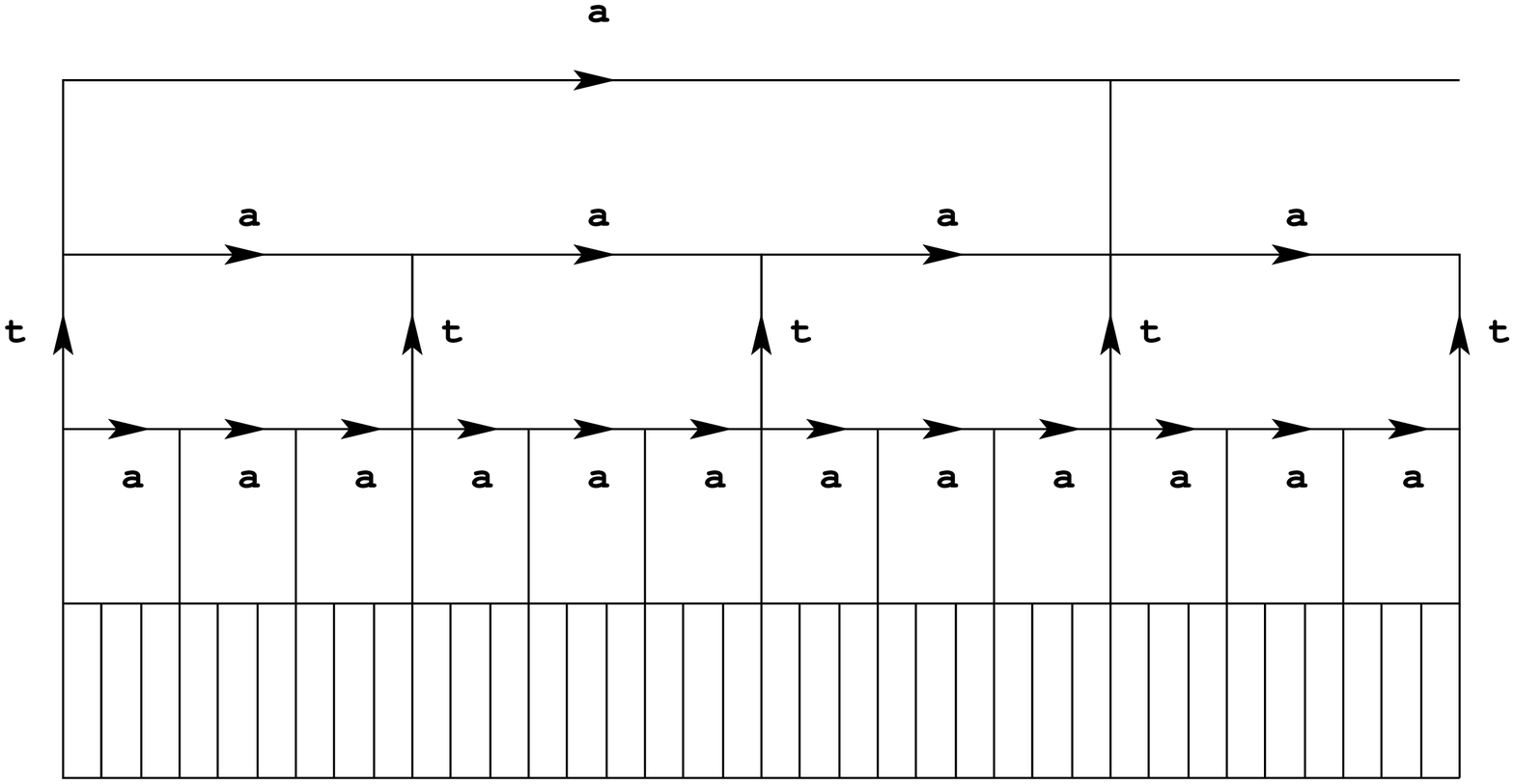} \\
 A brick & A sheet \\
\et\ec\caption{Parts of the Cayley graph of $G_3$}
\label{fig:brick-sheet}
\end{figure}
The complete Cayley graph is obtained by gluing these sheets together so that every vertex has degree four. From side-on the Cayley graph looks like a rooted $p$-ary tree. Some nice pictures can be found in \cite{\WordProc} pp.155-160.

We fix an orientation for the Cayley graph  by making $t$ edges go up vertically, and $a$ edges running horizontally from left to right.  Given this convention we can speak of the {\em top} or {\em bottom} of a brick. We  define the  {\em level} of a vertex in the Cayley graph to be the $t$-exponent sum of a word starting from the identity vertex to it.   So the identity is at level $0$, and $t^l$ is at level $l$. Note that this is well defined since if $u=_Gv$ then Britton's lemma implies $u$ and $v$ have the same $t$-exponent sum.

A word in the generators  $a^{\pm 1}, t^{\pm 1}$ is said to be of the form $P$ if it contains no $t^{-1}$ letters and at least one $t$ letter, and  of the form $N$ if it contains no $t$'s and at least one $t^{-1}$. Then a word is of the form $PNP$ say, if it is the concatenation of three words of the form $P,N,P$ in that order.  The $t$-exponent sum of a word is the number of $t$ letters minus the number of $t^{-1}$ letters.
We write $=_G$ when two words represent the same element in the group and $=$ when they are identical as strings, and $\ell(w)$ is the number of letters in the string $w$.

The following simple lemma and corollaries come from \cite{\Miller} and \cite{\Groves}.
\begin{lem}[Commutation]\label{lem:comm}
If $u,v$ have $t$-exponent sum zero 
then $uv=_Gvu$.
\end{lem}
\begin{cor}[Geodesics]\label{cor:NPNP}
A geodesic cannot contain a subword of the form $NPNP$ or $PNPN$.
\end{cor}
\begin{cor}[Pushing $a$s]\label{cor:pusha}
If $w$ is of type $NP$ and has $t$-exponent zero then $w=_Gu=t^{-k}u_P$ where $u_P$ is of type $P$ and $t$-exponent $k$, and $\ell(u)\leq \ell(w)$. If $w$ is of type $PN$ and $t$-exponent zero then $w=_Gv=v_Pt^{-k}$ where $v_P$ is of type $P$ and $t$-exponent $k$, and $\ell(v)\leq \ell(w)$.
\end{cor}

The two corollaries are simply a matter of commuting subwords of $t$-exponent zero past each other. We will show how this can be done in linear time and $O(n\log n)$ space in the algorithm. The trick is to use {\em pointers}, which we will explain in Section \ref{sect:partone} below.

\section{$t$-exponent sum of the input word}

The algorithm we describe in this paper applies only to input words with  non-negative $t$-exponent sum. To convert words of negative $t$-exponent sum to a geodesic, we modify the procedure given here as follows. Take as input the {\em inverse} of the input word, which has positive $t$-exponent sum. Run the algorithm as described on this word, then at the end, write the inverse of the  output word as the final output. 

Rewriting the input and output words as their inverses clearly can be done in linear time and space.

\section{Algorithm part one}\label{sect:partone}

The first stage of the algorithm is to rearrange the input word and freely reduce, to convert it to a standard form.  We assume  the input word has length $n$ and has non-negative $t$-exponent sum.

\begin{prop}\label{prop:partone}
Any word $w\in G_p$ of length $n$ with non-negative $t$-exponent sum can be converted to a word $u=_G w$ of the form $$u=t^{-k} a^{\epsilon_0} t\ldots ta^{\epsilon_q}t^{-m}$$ such that \bi\item$\ell(u)\leq n$, 
\item$k,q,m\geq 0$, 
\item $q\geq k+m$,
\item $|\epsilon_0|>0$ if $k>0$,  
\item $|\epsilon_q|>0$ if $m>0$,  
\item  $|\epsilon_i|<p$ for $0\leq i<q$,
\item $|\epsilon_q|<3p$, \ei
and moreover this can be achieved  in linear time and $O(n\log n)$ space.
\end{prop}

We prove this by describing a procedure to make this conversion.

Construct a list we  call List A of $n+2$ 5-tuples, which we view as an $5\times(n+2)$ table.  Each {\em address} in the table will contain either a blank symbol, an integer (between $-n$ and $n+1$, written in binary), or the symbol $t,t^{-1},a,a^{-1}, S$ or $F$. We refer to an address by the ordered pair (row,column). Note the space required for List A is therefore $O(n\log n)$ since entries are  integers in binary or from a fixed alphabet.
\bi\item Write the numbers 0 to $n+1$ in the first row. These entries will stay fixed throughout the algorithm. 
\item Row 2 will store the input word. Write $S$  for {\em start} at address (2,0),  then the input word letter by letter in addresses $(2,1)$ to $(2,n)$, and at address $(2,n+1)$ write   $F$ for {\em finish}. 
As the algorithm progresses, these entries will either remain in their original positions, or be erased (and replaced by a blank symbol). $S$ and $F$ are never erased. 
\item Row 3 will contain no entries at the beginning. As the algorithm progresses we will use the addresses in this row to store integers (between $-n$ to $n$).
 \item Write the numbers 1 through $n+1$ in the first $n+1$ addresses of row 4. Leave the final address blank. This row will act as a {\em pointer} to the next column address to be read. As the algorithm progresses, the entries in this row may change. 
 \item  In row 5, write a blank symbol in the first address, then write the numbers 0 through $n$ in the remaining addresses. This row indicates the  previous column address that was read (so are ``backwards pointers"). As the algorithm progresses, the entries in this row may change. 
\ei

Here is List A in its initial state, with input word $at^2a\ldots at^{-1}$.
\[\ba{l}\mathrm{List}\; \mathrm{A} \hspace{6mm} \downarrow
\\
\ba{|c|c|c|c|c|c|c|c|c|c|}
\hline
\mathrm{column} & 0 & 1 & 2 & 3 & 4 & \;\;\;\;\;\;\; \cdots \;\;\;\;\;\;\;  &  n-1 & n & n+1\\
\hline
\mathrm{word} &  S & a & t & t & a & \;\;\;\;\;\;\; \cdots \;\;\;\;\;\;\; & a &  t^{-1} & F \\
 \hline
t-\mathrm{exp} &  &  &  &  &  &  &  &   &  \\
 \hline
\mathrm{to} & 1 & 2 & 3 & 4 & 5 & \;\;\;\;\;\;\; \cdots \;\;\;\;\;\;\; & n &  n+1 &  \\
 \hline
\mathrm{from} &  & 0 & 1 & 2 & 3 & \;\;\;\;\;\;\; \cdots \;\;\;\;\;\;\; & n-2 &  n-1 & n \\   
\hline
\ea\ea\]
As the algorithm progresses, we will ``reorder'' the word written in row 2 using the pointers in rows 4 and 5 (and leaving the letters in row 2 fixed, possibly erasing some). To read the word, start at the $S$ symbol. Move to the column address indicated in row 4. At the beginning this will be column 1. From the current column, read the entry in row 4 to move the the next column. Continue until you reach the $F$ symbol. 
At any stage, to step back to the previous address, go to the column address indicated by row 5.
Throughout the algorithm, the pointers will never point to or from a column which has a blank symbol in row 2. The pointers  allow us to rearrange and delete letters from the word in row 2 efficiently (in constant time), without having to move any letters on the table.

For convenience, we indicate the current address being read by a {\em cursor}. We assume that moving the cursor from one position in the list to another takes constant time on a random access machine.

Here are two subroutines that we will use many times. Each one takes constant time to call.

\subsection*{Subroutine 1: Free reduction}
This subroutine eliminates freely canceling pairs $xx^{-1}$ in row 2 of List A, in constant time. Assume that the cursor is pointing to column $k$, and that the entry in the address (2, $k$) is not blank.

Read the entries in rows 2,4 and 5 of column $k$.
Say the letter in row 2 is $x$, and the integers in rows 4 and 5 are $i,j$.

If position $j$ row 2 is $x^{-1}$, then we can cancel this pair of generators from the word as follows:
\bi
\item read the integer in row 5 position $j$, and go to the address indicated (say it is $r$). In row 4 of this address, write $i$.
In row 5 position $i$, write $r$.
\item erase entries in columns $j$ and $k$
\item go to position $i$.
\ei
In this way, we have deleted $x^{-1}x$ from the word, and adjusted the pointers so that they skip these positions. 
\[\ba{ll}  \mathrm{List}\; \mathrm{A} \hspace{40mm} \downarrow 
\\
 \ba{|c|c|c|c|c|c|c|c|c|c}
\hline
\mathrm{column} & ... & r &...& j & ...  & k  &  ...   & i & ...   \\
\hline
\mathrm{word} & && &x^{-1} && x&&& \\
 \hline
t-\mathrm{exp} & &  &  &  &  &  &  &   &  \\
 \hline
 \mathrm{to} &&&&&& i&&&\\
 \hline
 \mathrm{from} & &  &  & r &  & j &  &  &  \\   
\hline
\ea\ea\]
\[
\leadsto 
\ba{l}\mathrm{List}\; \mathrm{A} \hspace{49mm} \downarrow
\\
\ba{|c|c|c|c|c|c|c|c|c|c}
\hline
\mathrm{column} &  ... & r & ...  & j &  ...   & k  &  ...   & i & ...   \\
\hline
\mathrm{word} & && &\times && \times&&& \\
 \hline
t-\mathrm{exp} & && &\times && \times&&& \\
 \hline
\mathrm{to} & & i& &\times && \times&&& \\
 \hline
\mathrm{from} &&& &\times && \times&&r& \\
\hline
\ea\ea\]

Else, if position $i$ row 2 is $x^{-1}$, we perform a similar operation to erase $xx^{-1}$ from the word, adjusting pointers appropriately.

Assuming that we can access positions using the pointers in constant time (that is, we have a random access machine), then this procedure takes constant time to run.

\subsection*{Subroutine 2:  Consecutive $a$s}
This subroutine eliminates the occurrence of subwords $a^{\pm 3p}$ (where $p\geq 2$ is fixed) in constant time.  Bounding the number of consecutive $a$ and $a^{-1}$ letters will be important for the time complexity of the algorithm later on.

 Again, assume the  cursor is pointing to column $k$  and the entry at address (2,$k$)  is not blank.
\bi\item If the letter at this address is $a$, set a counter $\mathtt{acount=1}$. Move back one square (using pointer in row 5) to column $j$. If address (2,$j$) is $a$, increment $\mathtt{acount}$. Repeat until $\mathtt{acount=3p}$ or the next letter is not $a$. Note maximum number of steps is $3p$ (constant). 

If $\mathtt{acount=3p}$ and you are at column $i$, 
write $ta^3t^{-1}$ over the first 5 $a$s, and blank symbols in the remaining $3p-5$ addresses up to position $k$. Adjust the pointers  so that the pointer at the added $t^{-1}$ points to the  value indicated at (4,$k$), and write the appropriate value in row 5 of that position.
\item If the letter at this position is $a^{-1}$, do the same with $a^{-1}$ instead of $a$.
\ei

This procedure takes constant time, and if it is {\em successful} (that is, replaces $a^{\pm 3p}$ by $t^{-1}a^{\pm 3} t$) it strictly reduces the length of the word.

We are now reading to describe part one of the algorithm, proving Proposition \ref{prop:partone}.

\subsection*{Step 1} Write the input word in freely reduced form on List A as follows. Read the first letter and write it in address $(1,2)$ of List A. For each subsequent letter if it freely cancels with the previous letter in row 2, erase the previous letter and continue. At the same time record the successive $t$-exponent sum of the word by incrementing and decrementing a counter each time a $t^{\pm 1}$ is read.

So the word in row 2 of the tape is freely reduced and has nonnegative $t$-exponent (by assumption). Fill in rows 1,4 and 5 of List A with the column numbers and pointers set to the initial state.

\subsection*{Step 2}  
In this step we eliminate all occurrences of $a^{\pm 3p}$ in the word.

Set $k=3p-1$.  Assume the entire word in row 2 is freely reduced, and  contains no more than $3p-1$ consecutive $a$s or $a^{-1}$s up to column $k$.

\bi\item Move cursor to column $k$. Following the pointer in row 4 move to the next column after $k$. If the letter in row 2 is $a^{\pm 1}$ then  perform Subroutine 2. 
\item
If the subroutine  finds $a^{\pm 3p}$, then with the cursor at the each end of the  inserted word (of length 5), perform Subroutine 1. Repeat until Subroutine 1 finds no more canceling pairs and so the entire word in row 2 is freely reduced. 
Set $k=$ the column to the right of the previous $k$ and repeat.
\ei

At the end of this procedure, the entire word is freely reduced and has no $a^{\pm 3p}$ subwords. The number of times Subroutine 2 is performed is at most the number of times we iterate the above steps, which is at most $n-3p$, and the total number of times we perform Subroutine 1 is $O(n)$ since each time it is successful the word reduces length, so it is successful at most $n$ times and unsuccessful at most twice (for each end) after each application of Subroutine 2.

So we now have a freely reduced word with less than $3p$ $a^{\pm 1}$ letters in succession, in row 2 of List A. The pointers in row 4 still point to columns to the right, since we have not commuted any subwords yet.

\subsection*{Step 3} 
Construct a second list  we call List B of $2n+1$ 4-tuples, which we view as a $4\times (2n+1)$ table. 
In the first row write the integers from $-n$ to $n$. 

Starting at column 0 of List A, set a counter \texttt{texp}=0. Reading the word in row 2 from left to right, if in column $k$  you read a $t^{\pm 1}$ letter, add $\pm 1$ to \texttt{texp}, and write the value of  \texttt{texp} at address $(3,k)$ of List A. In List B, if address (2,\texttt{texp}) is blank,  write $k$. If
it contains a value, write $k$ in address (3,\texttt{texp}) if it is blank, otherwise in  address (4,\texttt{texp}).

In other words, each time you read a $t^{\pm 1}$, write the current $t$-exponent sum underneath it, and in List B keep a record of how many times this $t$-exponent has appeared (which we call the number of ``{\em strikes}'' for that exponent) and at which positions in List A it appeared.

 Here we show List A for the input word $a t^2 a t a t^{-1} a t^{-1} a t^{-1} a t a  \ldots$, and the corresponding List B, as
an example.
\[\ba{ll}  \mathrm{List}\; \mathrm{A} \hspace{89.5mm} \downarrow 
\\
\ba{|c|c|c|c|c|c|c|c|c|c|c|c|c|c|c|c|c}
\hline
\mathrm{column} & 0 & 1 & 2 & 3 & 4 & 5 & 6 & 7 & 8 & 9 & 10 & 11 & 12 & 13 & 14  & \cdots\\
\hline
\mathrm{word} & S & a & t & t & a & t & a & t^{-1} & a & t^{-1} & a & t^{-1} & a & t & a  & \\
 \hline
t-\mathrm{exp} &  &  & 1  & 2 &  & 3 &  & 2  & & 1& & 0 & & 1 & & \\
 \hline
\mathrm{to} & 1 & 2 & 3  & 4 & 5 & 6 & 7 & 8  &9 & 10& 11& 12 & 13& 14 &15 & \\
 \hline
\mathrm{from} &  &  0& 1  & 2 &   3 & 4 & 5  &6 & 7& 8& 9 &10 & 11 & 12&13 & \\
\hline
\ea\\
\\
\mathrm{List}\; \mathrm{B}\\
\ba{|c|c|c|c|c|c|c|c|c|c|c|c|c|c|c|c}
\hline
t-\mathrm{exp} & \cdots & -6 & -5  & -4 & -3 & -2 & -1 & 0  & 1& 2& 3& 4 & 5& 6 & \cdots\\
 \hline
\mathrm{strike}\;1&  &  &   &  &  &  &  & 11  & 2& 3& 5&  & &  &  \\
 \hline
\mathrm{strike}\;2 &  &  &   &  &  &  &  &  & 9& 7& &  & &  &  \\
\hline
\mathrm{strike}\;3  &  &  &   &  &  &  &  &  & 13& & &  & &  &  \\
\hline
\ea\ea\]

When an entry occurs in the last row of List B at some position labeled \texttt{texp}, meaning the same exponent has occurred 3 times , then we have a prefix of the form either $NPNP$ or $PNPN$, so we  apply Corollary \ref{cor:NPNP} as follows.
Suppose the entries in this column are $p_a,p_b,p_3$, with $p_3$ the most recently added.
These correspond to the positions in List A where the value \texttt{texp} have appeared.

To begin with, the word written in row 2 of List A appears in the correct order (from left to right), the pointers have only been used to possibly skip blank addresses. So at the start of this step we know that $p_a$ comes before $p_b$. However, 
as the algorithm progresses, we will not know which of $p_a$ and $p_b$ comes first in the word. That is, as we introduce pointers to List A to move subwords around, a letter in column $p$ could sit before a letter in column $q$ with $q<p$. We do know that $p_3$ is the right-most position.

The word read in its current order is either $\ldots p_a \ldots p_b \ldots p_3\ldots $ or \\
$\ldots p_b \ldots p_a \ldots p_3\ldots $.
We can determine the order with the following subroutine.

\subsection*{
Subroutine 3: Determine order of $p_a,p_b$.}

Starting at $p_a$, scan back (using pointers in row 5) through the word to the position of the previous $t^{\pm 1}$ letter, or the $S$ symbol.  Since we have at most $3p-1$ consecutive $a^{\pm 1}$ letters,  this takes constant time. Do the same for $p_b$.

If we come to $S$ from either $p_a$ or $p_b$, then we know that this position must come first.

If both $p_a,p_b$ are preceded by  $t^{\pm 1}$ letters, then we need more information.  Start at $p_a$ and scan forward to the first $t^{\pm 1}$, whose  position we call $q_a$. Start at $p_b$ and scan forward to the first $t^{\pm 1}$, call this position $q_b$. This takes constant time since there are at most $3p-1$ consecutive $a^{\pm 1}$ letters.

 Now one of  columns $q_a,q_b$ must contain a $t^{\pm 1}$ in row 2, with sign  opposite to that of $p_3$.

\bi\item If $q_a$ is same sign as $p_3$, then order must be $p_a-q_a-p_b-q_b-p_3$

\item If $q_b$ is same sign as $p_3$, then order must be $p_b-q_b-p_a-q_a-p_3$

\item Both  $q_a,q_b$ have opposite sign to  $p_3$.
In this case, we look at the letters in row 2 of columns  $p_a, p_b$. If the letter at address $(2,p_a)$ has opposite sign to that in$(2,p_3)$, then it must come first, since one of $p_a,p_b$ must match up with $q_a,q_b$.

If both $p_a,p_b$ have same letter as $p_3$ in row 2, then we are in a situation like $tat^{-1}atat^{-1}at$. But since there is a $t^{\pm 1}$ letter preceding both $p_a$ and $p_b$, then the $t$-exponent before $p_a,p_b,p_3$ are read is the same, and is recorded three times. This case cannot arise since we apply this procedure the first time we see the same number more than twice.

\ei

Using this subroutine we can determine the correct order of the columns $p_a,p_b$ and $q_a,q_b$, in $O(n)$ time.
Rename the first position $p_1$ and second $p_2$, and $q_1,q_2$ as appropriate.
So we have $p_1-q_1-p_2-q_2-p_3$

The subword between positions $q_1$ and $p_2$ has $t$-exponent 0, as does the subword from $q_2$ to $p_3$. By commuting one of these subwords (using Lemma \ref{lem:comm}) we can place a $t$ next to a $t^{-1}$ somewhere and get a free cancellation. The precise instruction will depend on the letters at each of these addresses, and we will consider each situation case-by-case.

\subsection*{Case 1}
\[\ba{ll}  \mathrm{List}\; \mathrm{A}  
\\
\ba{|c|c|c|c|c|c|c|c|c|c|c|c}
\hline
\mathrm{column} & \;\cdots \;& p_1 &\; \cdots \; & q_1 &\; \cdots \;& p_2  & \; \cdots \;   & q_2 & \;\cdots\; & p_3&\; \cdots \; \\
\hline
\mathrm{word} & & t & & t^{-1} &   & t  &    & t^{-1} && t &\\
 \hline
t-\mathrm{exp} & & k_1 & & k_2 &   &   &    &  &&  &\\
 \hline
\mathrm{to} & & i_1 & & i_2 &   & i_3  &    &  &&  &\\
 \hline
\mathrm{from} & & j_1 & & j_2 &   & j_3  &    &  &&  &\\
\hline
\ea\ea\]

Between $p_1$ and $q_1$ we have only $a$ letters (or nothing).
So we will commute the subword $q_1-p_2$ back towards $p_1$ as follows: 
\bi\item $j_1$ row 4, replace $p_1$ by $i_2$
\item $i_2$ row 5, replace $q_1$ by $j_1$

\item $p_2$ row 4, replace $i_3$ by $i_1$
\item $i_1$ row 5, replace $p_1$ by $p_2$

\item $j_2$ row 4, replace $q_1$ by $i_3$
\item $i_3$ row 5, replace $p_2$ by $j_2$
\item delete columns $p_1,q_1$
\item delete $p_1$ and $q_1$ from List B columns $k_1$ and $k_2$ respectively.
\ei
This has the effect of moving the subword back through the word, but without changing more that a constant number of entries in the lists (with random access).

The next case applies to our running example shown above.

\subsection*{Case 2} 
\[\ba{ll}  \mathrm{List}\; \mathrm{A}
\\
\ba{|c|c|c|c|c|c|c|c|c|c|c|c}
\hline
\mathrm{column} & \;\cdots \;& p_1 &\; \cdots \; & q_1 &\; \cdots \;& p_2  & \; \cdots \;   & q_2 & \;\cdots\; & p_3&\; \cdots \; \\
\hline
\mathrm{word} & & t & & t &   & t^{-1}  &    & t^{-1} && t &\\
 \hline
t-\mathrm{exp} & & k_1 & &  &   &   &    & k_2 &&  &\\
 \hline
\mathrm{to} & & i_1 & &  &   &   &    &  i_2&&  i_3&\\
 \hline
\mathrm{from} & & j_1 & &  &   &   &    &  j_2&&  j_3&\\
\hline
\ea\ea\]

This time we will  commute the subword $q_2-p_3$ back past the subword of $t$-exponent zero and next to $p_1$ as follows:
\bi\item $j_1$ row 4, replace $p_1$ by $i_2$
\item $i_2$ row 5, replace $p_1$ by $j_1$

\item $p_3$ row 4, replace $i_3$ by $i_1$
\item $i_1$ row 5, replace $p_1$ by $p_3$

\item $j_2$ row 4, replace $q_2$ by $i_3$
\item $i_3$ row 5, replace $p_3$ by $j_2$

\item delete columns $q_1,p_2$
\item delete $p_1$ and $q_2$ from List B columns $k_1$ and $k_2$ respectively.
\ei

Below we show the two lists after commuting and deleting $tt^{-1}$ for our running example after this step. We read the new word off List A following the pointers as $a\leadsto at \leadsto tatat^{-1}at^{-1}a \leadsto at^{-1}\ldots$.

\[\ba{ll}  \mathrm{List}\; \mathrm{A} \hspace{92.5mm} \downarrow 
\\
\ba{|c|c|c|c|c|c|c|c|c|c|c|c|c|c|c|c|c}
\hline
\mathrm{column} & 0 & 1 & 2 & 3 & 4 & 5 & 6 & 7 & 8 & 9 & 10 & 11 & 12 & 13 & 14 &  \cdots\\
\hline
\mathrm{word} & S & a & \times& t & a & t & a & t^{-1} & a & t^{-1} & a & \times & a & t & a &  \\
 \hline
t-\mathrm{exp} &  &  & \times & 2 &  & 3 &  & 2  & & 1& & \times & & 1 & &  \\
 \hline
\mathrm{to} & 1 & 12 & \times  & 4 & 5 & 6 & 7 & 8  &9 & 10& 14& \times & 13& 3 &15 &  \\
 \hline
\mathrm{from} &  &  0& \times  & 13 &   3 & 4 & 5  &6 & 7& 8& 9 &\times & 1 & 12&10 &  \\
\hline
\ea\\
\\
\mathrm{List}\; \mathrm{B}\\
\ba{|c|c|c|c|c|c|c|c|c|c|c|c|c|c|c|c}
\hline
t-\mathrm{exp} & \cdots & -6 & -5  & -4 & -3 & -2 & -1 & 0  & 1& 2& 3& 4 & 5& 6 & \cdots\\
 \hline
\mathrm{strike}\;1&  &  &   &  &  &  &  &   & 9& 3& 5&  & &  &  \\
 \hline
\mathrm{strike}\;2 &  &  &   &  &  &  &  &  & 13& 7& &  & &  &  \\
\hline
\mathrm{strike}\;3  &  &  &   &  &  &  &  &  & & & &  & &  &  \\
\hline
\ea\ea\]

The remaining cases are similar and we leave it to the reader to imagine the instructions for each one. Corollary \ref{cor:NPNP} guarantees that some commutation will reduce length in each case.

With the cursor at position $p_3$ or if blank, the non-blank letter to its right, perform Subroutine 1 until unsuccessful, then Subroutine 2 until unsuccessful, and alternately until both are unsuccessful.
Since each successful application of a subroutine reduces word length, the total number of successes of each is $n$ throughout the whole algorithm. 
Once both are unsuccessful the entire word is again freely reduced and avoids $a^{\pm 3p}$, and the cursor is at the next non-blank  position to the right of $p_3$. We then resume Step 3 from this position.

So after performing this procedure, List A contains a possibly shorter word in row 2, which is read starting at column 0 and following pointers, and List B contains the correct data of $t$-exponents and addresses (although addresses don't stay in order).
Since we removed one of the 3 ``strikes'', we start at $p_3+1$ and continue filling out row 3 of List A, adding appropriate entries to List B, until we again get 3 strikes.
Note that we do not backtrack, so the total number of right steps  taken in Step 3 (assuming the random access model of computation allows us to read and write at any specified position in the table) is $O(n)$. The number of times we need to apply the subroutines (successfully and unsuccessfully) is also  $O(n)$ regardless of how many times Step 3 is called, so so all together this step takes  $O(n)$ time.

At the end of this step, since the word has nonnegative $t$-exponent sum, and all ``3 strikes'' have been eliminated, the word must be of the form  $E$, $P$, $PN$,  $NP$, $NPN$, or $PNP$.

\subsection*{Step 4}  
If the word at this stage is of the form $PNP$, its $t$-exponent sum must be zero, so it has a prefix of the form $PN$ and suffix $NP$, both of zero $t$-exponent sum. Say $p_1,p_2$ are the two positions that the $t$-exponent is zero. 
We commute the prefix and suffix by rewriting pointers.
If the word on List A is not written in order from left to right, we can create a new List A in which the word is in correct order, by reading the current list following the pointers. 

So after this we can assume 
the configuration of List A is as follows:
\[\ba{ll}  \mathrm{List}\; \mathrm{A} 
\\
\ba{|c|c|c|c|c|c|c|c|}
\hline
\mathrm{column} & 0 & \; \cdots \; & p_1 &\; \cdots \;& p_2  & \; \cdots \;   & n+1 \\
\hline
\mathrm{word} & S&  & t^{-1}  &    & t & & F\\
 \hline
t-\mathrm{exp} & &  & 0 &  &  0  &   &\\
 \hline
\mathrm{to} &  i_1 & & i_2& &  i_3& &\\
 \hline
\mathrm{from} & & & j_1  & & j_2 & &  j_3\\
\hline
\ea\ea\]

Then do the following:
\bi\item 0 row 4 replace $i_1$ by $p_2$
\item $p_2$ row 5 replace $j_2$ by 0

\item $j_3$ row 4 replace $n+1$ by $p_1$
\item $p_1$ row 5 replace 0 by $j_3$

\item $j_1$ row 4 replace $p_2$ by $n+1$
\item $n+1$ row 5 replace $j_3$ by $j_1$.
\ei
The word is now  the form $NPN$.

\subsection*{Step 5}  
At this point  the word in List A row 2  is of the form $E$, $P$, $PN$, $NP$ or $NPN$.
We can ascertain which of these it is in constant time simply by checking the first and last $t^{\pm 1}$ letter in the word, which lie at most $3p$ steps  from the ends of the tape (positions $S$ and $F$), following pointers.
\bi\item No $t^{\pm 1}$ letters: $E$
\item First $t$ last $t$: $P$
  \item First $t$ last $t^{-1}$: $PN$
  \item First $t^{-1}$ last $t$: $NP$  
  \item First $t^{-1}$ last $t^{-1}$: $NPN$
  \ei
  
  In the case $E$, the word is $a^i$ with $|i|<3p$, so by checking a finite list we can find a geodesic for it and be done.
 So for the rest of the algorithm assume $u$ is of the form $P,PN,NP,NPN$.

 In the last two cases $NP$ and $NPN$
 it is possible the word contains a subword of the form $t^{-1}a^{xp}t$ for an integer $x$, which will be in $\{\pm 1,\pm 2\}$. If so we want to replace it by $a^x$, which will always reduce length. This is easily done in constant time, assuming the cursor is pointing to the column containing  the first $t$ letter in the word. From this column scan back at most $3p$ letters to  a $t^{-1}$ letter and count the number of $a^{\pm 1}$ letters in between. Then if you find $t^{-1}a^{xp}t$ rewrite with $a^x$ and then move forward to the next $t$ letter. Repeat this at most $O(n)$ times 
(the maximum number of $t$ letters in the word) until no such subword appears. Note that to locate the first $t$ letter in the word to start this takes $O(n)$ time to scan the word, but you only need to do this once.

We now have a word in one of the forms $P,PN,NP,NPN$ which contains no   $t^{-1}a^{xp}t$ subword.

\subsection*{Step 6}  
In this step 
   we apply Corollary \ref{cor:pusha} to push all of the $a^{\pm 1}$ letters in the word  into a single sheet of the Cayley graph. The output of this step (the word in row 2 of List A) will be a word of the form $t^{-k}u_Pt^{-l}$ where $u_P$ is a word of type $P$ with $t$-exponent $\geq k+l$, and $k,l\geq 0$
  and the word still does not contain a subword of the form $t^{-1}a^{xp}t$.

\subsection*{Case $P$}  The word is of the form $u_P$ so done.

\subsection*{Case $PN$}
 Say \texttt{texp} is the final $t$-exponent of the word, which occurs at positions $p_1$ and $q_1$. If $q_1$ is not the end of the word (that is, there are $a^{\pm 1}$ letters at the end of the word), then we want to push the $a$ letters there back through the word to $t$-exponent 0, which starts after $p_1$. The configuration of the tape is as follows (where we assume $j_3\neq q_1$ since there are $a^{\pm 1}$ letters at the end of the word):
\[\ba{ll}  \mathrm{List}\; \mathrm{A}
\\
\ba{|c|c|c|c|c|c|c|c|}
\hline
\mathrm{column} &  \; \cdots \; & p_1 &\; \cdots \;& p_2  & \; \cdots \;   & n+1 \\
\hline
\mathrm{word} &  & t  &    & t^{-1} & & F\\
 \hline
t-\mathrm{exp}  &  &  \mathtt{texp} &  &  \mathtt{texp}  &   &\\
 \hline
\mathrm{to} &  &i_1 & & i_2&    &\\
 \hline
\mathrm{from} & &  j_1  & & j_2 & &  j_3\\
\hline
\ea\ea\]

Then do the following:
\bi\item $q_1$ row 4, replace $i_2$ with $n+1$
\item $n+1$ row 5, replace $j_3$ with $q_1$

\item $p_1$ row 4, replace $i_1$ with $i_2$
\item $i_2$ row 5, replace $q_1$ with $p_1$

\item $j_3$ row 4, replace $n+1$ with $i_1$
\item $i_1$ row 5, replace $p_1$ with $j_3$
\ei
So we have commuted the word $a^{\pm m}$ at the end, through the subword of $t$-exponent 0.
Check for cancellation of  $a^{\pm 1}a^{\mp 1}$, if this occurs then cancel. Repeat up to $3p-1$ (constant)  times.

Next, let $p_2,q_2$ be the columns at which  $(\mathtt{texp}+1)$ occurs. Repeat the procedure at this level. Again freely cancel.

Iterate this until all $a^{\pm 1}$ letters are pushed into the middle of the word, 
so the resulting word is of the form $u_Pt^{-k}$ where $u_P$ is a word of type $P$ with $t$-exponent at least $k$. Note that at the top level there is only one $a^i$ subword, which will not be canceled, so there is no free cancellation of $t$ letters in this step. At every other level there can be at most $6p-2$ consecutive $a^{\pm 1}$ letters.

\subsection*{Case $NP$} Same as previous case, this time pushing $a$s to the right.

\subsection*{Case $NPN$} Break the word into $NP$ and $PN$ subwords, with the $NP$ subword ending with $t$, and each of zero $t$-exponent sum,
 and perform the above steps to push $a$ letters to the right and left  respectively, then  ensure there is no $t^{-1}a^{ip}t$ subword. In the $NP$ prefix the 
maximum number of  consecutive $a^{\pm 1}$ letters is $6p-2$ at every level except the top level, since we cut the word immediately after a $t$, so at this level we have at most $3p-1$ consecutive $a^{\pm 1}$s, and in the $PN$ suffix we can have at most $6p-2$ at every level, so all together there could be at most $9p-3$ consecutive $a^{\pm 1}$s.

So after this step,  the positive part of the word, $u_P$,  stays within a single sheet of the Cayley graph. We write 
 \[u_P=a^{\epsilon_0}ta^{\epsilon_1}\ldots a^{\epsilon_{m-1}}ta^{\epsilon_m}\] where each $|\epsilon_i|<9p$  and with $\epsilon_0$ not a multiple of $p$ when the word is of the form $NP$ or $NPN$.

\subsection*{Step 7}
In this step we remove all occurrences of $a^{\pm p}t$ in the word. Scan to the first $t$ in row 2. If the preceding $p$ letters are $a^{\pm 1}$ then replace $a^{\pm p}t$ by $ta^{\pm 1}$. Stay at this $t$ letter and repeat until there is no  $a^{\pm p}t$, then move to the next $t$ letter. Since each replacement reduces length the time for this step is linear. At the end you have eliminated all  $a^{\pm p}t$ subwords so the word is of the form $u=t^{-k}u_Pt^{-l}=t^{-k}a^{\epsilon_0} t\ldots ta^{\epsilon_m}t^{-l}$ with $|\epsilon_i|<p$ for $i<m$. 

\subsection*{Step 8} Set $\epsilon_m=M_0$. If  $|M_0|<3p$ then stop,  part one of the algorithm is done. If $|M_0|\geq 3p$ then 
we will replace the last term  $a^{\epsilon_m}=a^{M_0}$, by a word of the form $a^{\eta_0} t\ldots ta^{\eta_s}t^{-s}$ with $|\eta_i|<p$ for $i<s$ and $|\eta_s|<3p$, as follows. 

\bi\item Go to the first $t^{-1}$ after $a^{M_0}$ on the tape, then scan back $3p$ steps. If you read $a^{\pm 3p}$ in these steps, then replace the subword by  $ta^{\pm 3}t^{-1}$, which strictly reduces length. Then scan back another $p$ steps from the $t$ letter you have inserted, and if you read $a^{\pm p}t$ then replace it by $ta^{\pm 1}$. Repeat until you don't read $p$ consecutive $a$s or $a^{-1}$s.

If you did any replacing, you now have a word of the form $u=t^{-k}a^{\epsilon_0} t\ldots a^{\epsilon_{m-1}}ta^{\eta_0}t a^{M_1}t^{-1}t^{-l}$ with $|\epsilon_i|<p$ for $i<m$ and $|\eta_0|<p$. The number of steps to do this is $O(|M_0|)$. Note that $|M_1|\leq |M_0|/p$.
\item Repeat the previous step, by scanning back $3p$ from the last $t^{-1}$ inserted. If you read $a^{\pm 3p}$ replace and repeat the procedure as before, and if $|M_i|<3p$ stop. Note that each $|M_i|\leq |M_{i-1}|/p$, so $|M_i|=\leq |M_0|/(p^i)$.
\ei
Each iteration of this takes  $O(|M_i|)=O(|M_0|/(p^i))$ steps, so in total the time for this procedure is  
 \[O(|M_0|+|M_0|/p+|M_0|/p^2+\ldots) =O(|M_0|)=O(n)\]
by the geometric series formula.

So the word on the tape is now of the form $u=t^{-k}a^{\epsilon_0}t\ldots a^{\epsilon_q}t^{-m}$ where  $|\epsilon_i|<p$ for $0<i<q$, $|\epsilon_0|>0$ if $k>0$,  $|\epsilon_q|<3p$, and since all steps together took no more than linear time and space, we have proved Proposition \ref{prop:partone}.

\section{Algorithm part 2}

The second part of the algorithm finds a geodesic for the output word $u$ of part one. 
We will show that such a geodesic can be found in the same sheets of the Cayley graph as $u$, and moreover stays close to $u$ in a certain sense, so we can compute it in bounded time and space.

For ease of exposition we treat the four possible cases -- $u$ of type $P$, $NP$, $PN$, $NPN$ -- separately.

\begin{prop}\label{prop:twotypeP}[$u$ of type $P$]
Let $u=a^{\epsilon_0}t\ldots ta^{\epsilon_q}$ be the output of part one, where $|\epsilon_i|<p$ for $0\leq i<q$,  $|\epsilon_q|<3p$, and $\ell(u)\leq n$, the length of the initial input word.

Then    there is a geodesic $v$ for $u$ of the form $v=a^{\eta_0} t\ldots ta^{\eta_q}$  or  $v=a^{\eta_0} t\ldots ta^{\eta_q}ta^{\rho_{q+1}}t\ldots ta^{\rho_{q+s}}t^{-s}$ with $s> 0$,
which can be  computed in linear time and $O(n\log n)$ space.
\end{prop}

\begin{proof}
Let $v$ be some geodesic for $u$. Applying Proposition \ref{prop:partone} we can put $v$ into the form $t^{-d} a^{\psi_0} t\ldots ta^{\psi_r}t^{-l}$ 
for some integers $d,r,l$, and $\psi_0\neq 0$ if $d>0$,  $|\psi_i|<p$ for $0\leq  i<r$, $|\psi_r|<3p$ and $|\psi_r|>0$ if $s>0$.

If $d>0$ then applying Britton's Lemma to $u^{-1}v= a^{-\epsilon_q}t^{-1}\ldots t^{-1}a^{-\epsilon_0} t^{-d} a^{\psi_0} t\ldots ta^{\psi_r}t^{-l}$ we must have  a pinch  $ t^{-1}a^{-\psi_0} t$, which is not possible since $0<|\psi_0|<p$. So $d=0$ and  $v=a^{\psi_0} t\ldots ta^{\psi_r}t^{-l}$. Since $u,v$ have the same $t$-exponent sum we also have $q=r-l$, so we can write  $v=a^{\eta_0} t\ldots ta^{\eta_q}$  if $l=0$, or  $v=a^{\eta_0} t\ldots ta^{\eta_q}ta^{\rho_{q+1}}t\ldots ta^{\rho_{q+s}}t^{-s}$ with $s> 0$. 

This proves the first assertion. Now to compute it. Observe that in the Cayley graph for $G_p$, the paths $u,v$ travel up a single sheet, from level $0$ to level $q$. The suffix of $v$ may continue to travel up, and return to level $q$ via $t^{-s}$, so $v$ lies in precisely the sheet determined by $u$.

We locate  $v$ by tracking the path $u$, and all possible geodesic paths, level by level in this sheet, as follows.

Label the identity vertex by $S$ (for {\em start}). If $\epsilon_0\neq 0$,  draw a horizontal line of  $|\epsilon_0|$ $a$ edges,  to the left if $\epsilon_i<0$ and right if positive. Then draw a vertical $t$ edge up from this line, and complete the picture by drawing in the brick containing  $a^{\epsilon_0}t$ on its boundary.  Label the corner corresponding to the endpoint of $a^{\epsilon_0}t$ by $U$, and on each corner compute the distance $d_1,d_2$ back to $S$ (which will be $|\epsilon_0|+1$ and $p-|\epsilon_0|+1$).  If $|d_1-d_2|\leq 1$ then keep both labels, and store two words $g_1,g_2$ which are geodesics to these points. If their difference is greater than 1 then discard the larger label and only keep the short one, plus a geodesic word $g_1$ to it. See Figure \ref{fig:parttwo1}. If $\epsilon_0=0$ then simply draw the $t$ edge and no bricks, and store $g_1=t$.
\begin{figure}[ht!]
\bc
\psfrag{U}{$U$}\psfrag{3}{$3$} \psfrag{0}{$S$}  
\includegraphics[scale=.34]{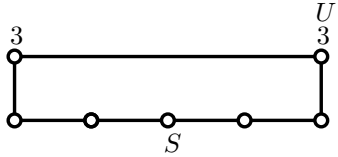} 
\ec\caption{First level. In these figures we are in the group $G_4$.}\label{fig:parttwo1}
\end{figure}

Now assume you have drawn this picture up to level $i<q$, so you have $i$ bricks stacked on top of each other vertically, and the top brick has its top corner(s) labeled $U$ corresponding to the endpoint of $a^{\epsilon_0}t\ldots a^{\epsilon_{i-1}}t$, and $d_1,(d_2)$ the shortest distance(s) back to $S$. Also you have stored geodesic(s) $g_1,(g_2)$ to the points labeled $d_1,d_2$.

From the point $U$, draw a horizontal line for $a^{\epsilon_i}$ to the left or right depending on the sign. Then draw a vertical $t$ edge up.
Now since $|\epsilon_i|<p$, the brick with boundary  $a^{\epsilon_i}t$ also contains the point(s) $d_1,(d_2)$, and so to compute the distance to the corners of the new brick, one simply computes from these points, since they are the closest points on the level $i$ in this sheet. Label the corners of the new brick in level $i+1$ by $U,d_1,(d_2)$ as before. Update $g_1,(g_2)$ by appending suffix(es) $a^jt$.  Once we compute the data for some level, we can discard the data for previous level.

\begin{figure}[ht!]
\bc
\psfrag{U}{$U$}\psfrag{6}{$6$}\psfrag{5}{$5$}\psfrag{3}{$3$}   \psfrag{8}{$8$}  
\bt{ccc}
\includegraphics[scale=.34]{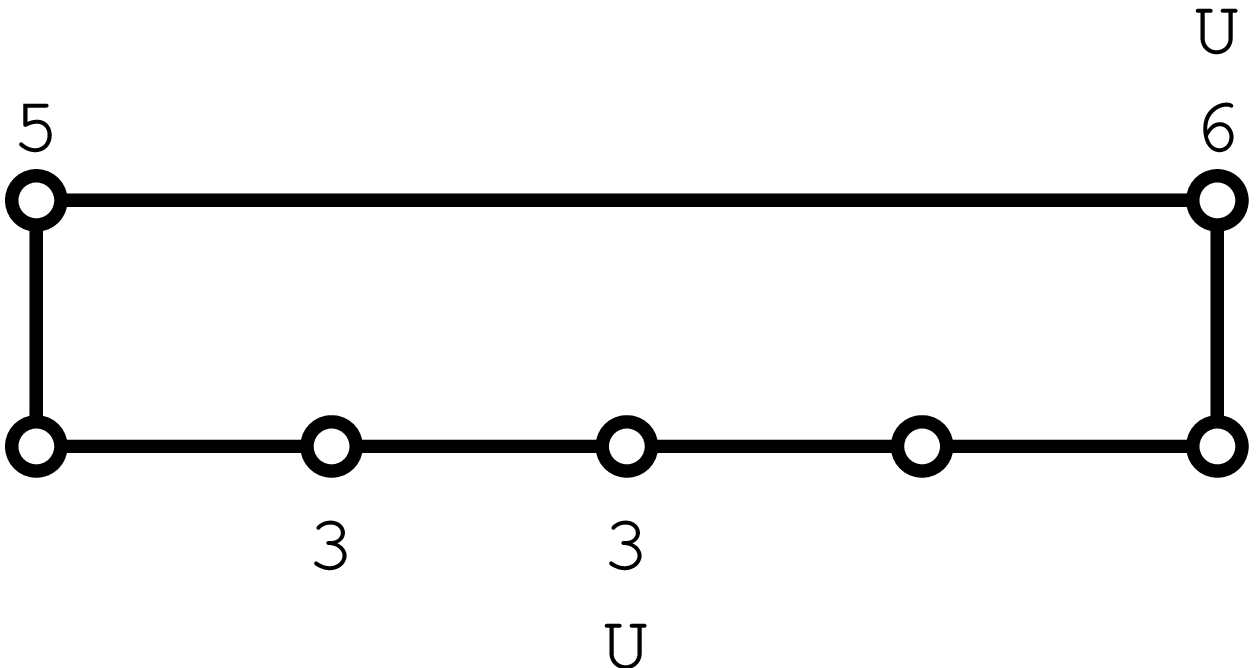} & $\;\;\;\;\;\;\;\;\;\;\;\;\;\;\;\;\;\;$ & \includegraphics[scale=.34]{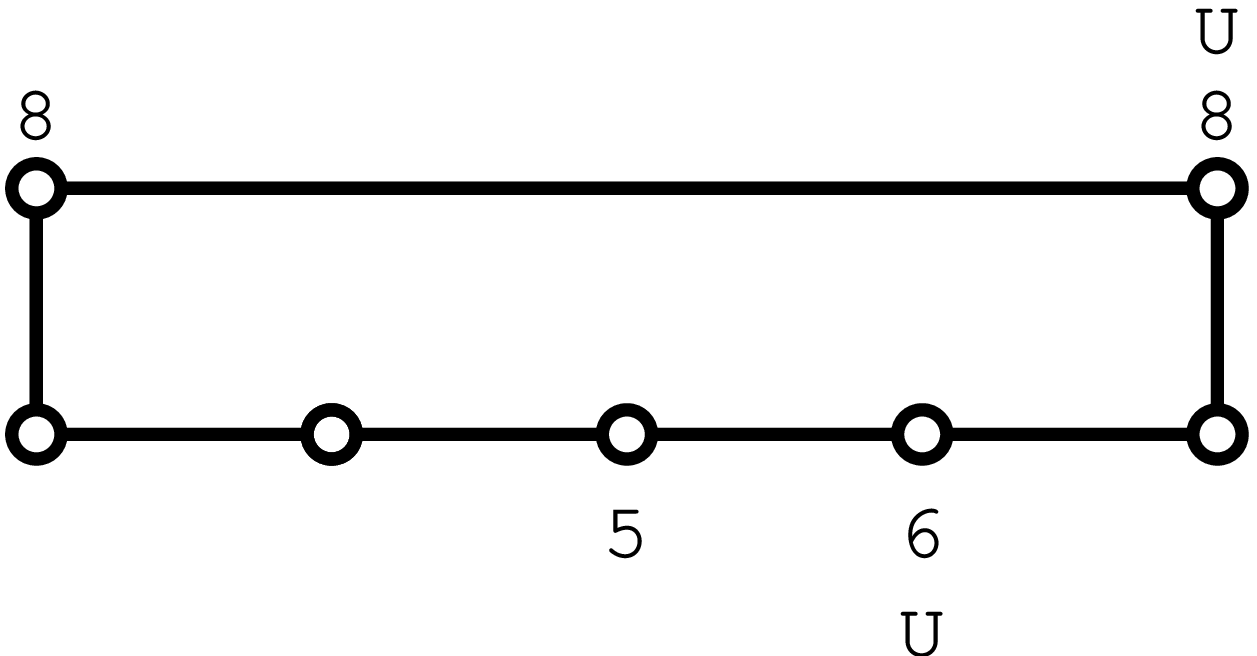} \\
{Level 2} && {Level 3}\et
\ec\caption{The next two levels}\label{fig:AA}
\end{figure}
In this way one can draw the path $u$ in its sheet up to level $q$, and keep track of the distances from $S$ to each level of the sheet, using constant time and $O(n\log n)$ space (writing the labels in binary) for each level. Figure \ref{fig:AA}  shows the next two iterations of this.

At level $q$, draw $a^{\epsilon_q}$ from $U$ to the endpoint of $u$, which we mark with $E$. Note this distance is at most $3p-1$.

Now, a geodesic to $E$ from $S$ must travel the shortest distance up from level $0$ to this level, so without loss of generality $v$ starts with one of $g_1$ or $g_2$, say $g_x$ ($x=1,2$) ending at the point labeled $d_x$, then ends with a suffix from $d_x$ to $E$ of the form $a^m$, or of type $PN$, which by Corollary \ref{cor:pusha} we may assume  has the form $ta^{\rho_{1}}t\ldots ta^{\rho_{s}}t^{-s}$ with $s> 0$,

Since  $d(E,U)<3p$ and $d(U,d_x)\leq 1$ then this suffix is equal to a word in $a^j$ or length at most $3p$, which is a fixed constant, so finding a geodesic suffix for $v$ is simply 
 a matter of checking a finite number of possible suffixes, which can be done in constant time, so we are done. See Figure \ref{fig:top}.
\begin{figure}[ht!]
\bc
\psfrag{U}{$U$}\psfrag{d1}{$d_1$}  \psfrag{d2}{$d_2$}   \psfrag{E}{$E$} 
\includegraphics[width=11.8cm]{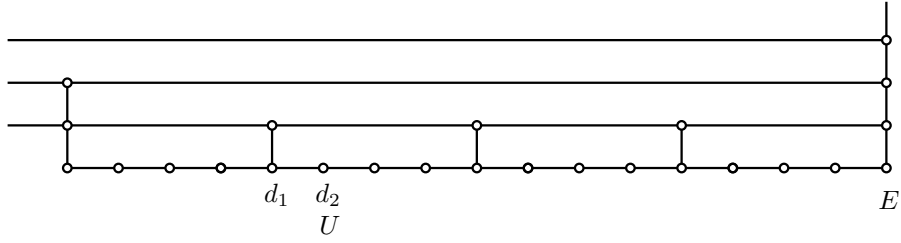} 
\ec\caption{The top level}\label{fig:top}
\end{figure}
\end{proof}

\begin{prop}\label{prop:twotypeNP}[$u$ of type $NP$]
Let $u=t^{-k}a^{\epsilon_0}t\ldots ta^{\epsilon_q}$ be the output of part one, where $q\geq k$, $|\epsilon_i|<p$ for $0<i<q$ and $|\epsilon_q|<3p$, and $\ell(u)\leq n$, the length of the initial input word.

Then    there is a geodesic $v$ for $u$ of the form $v=t^{-k}a^{\eta_0} t\ldots ta^{\eta_q}$  or  $v=a^{\eta_0} t\ldots ta^{\eta_q}ta^{\rho_{q+1}}t\ldots ta^{\rho_{q+s}}t^{-s}$ with $s> 0$. 
which can be  computed in linear time and $O(n\log n)$ space.
\end{prop}

\begin{proof}
Repeat the previous proof, this time Britton's lemma applied to $u^{-1}v$ implies that $v$ has the form  $v=t^{-k}a^{\eta_0} t\ldots ta^{\eta_q}$   or \\
 $v=$ $t^{-k}a^{\eta_0} t\ldots ta^{\eta_q}ta^{\rho_{q+1}}t\ldots ta^{\rho_{q+s}}t^{-s}$ with $s> 0$. 

So $u,v$ have identical $t^{-k}$ prefixes, after which both paths travel from level $-j$ up to level $q\geq 0$, with $v$ possibly traveling further up then back to level $q$.

The algorithm to find $v$ is identical if we make $S$ the label of the endpoint of the prefix $t^{-k}$ instead of the identity element.
\end{proof}

The final two cases are only slightly more involved than these cases.
The difference here is that the word $v$ may not go up as high as $u$. As an instructive example, suppose $u=(a^{1-p}t)^nat^{-n}$ ($p\geq 2$), which is in the form out the output of part one. This word has geodesic representative $a$, and so the geodesic for it no longer stays close. In spite of this we have the following.

\begin{prop}\label{prop:twotypePN}[$u$ of type $PN$]
Let $u=a^{\epsilon_0}t\ldots ta^{\epsilon_q}ta^{\delta_{q+1}}t\ldots ta^{\delta_{q+r}}t^{-r}$ be the output of part one, where $r\geq 1$, $|\epsilon_i|<p$ for $0\leq i\leq q$, $|\delta_{q+i}|<p$ for $1\leq i<r$, $0<|\delta_{q+r}|<3p$, and $\ell(u)\leq n$, the length of the initial input word. Note that $u$  ends at level $q$.

Then    there is a geodesic $v$ for $u$ of the form $v=a^{\eta_0} t\ldots ta^{\eta_q}$  or  $v=a^{\eta_0} t\ldots ta^{\eta_q}ta^{\rho_{q+1}}t\ldots ta^{\rho_{q+s}}t^{-s}$ with $s> 0$,
which can be  computed in linear time and $O(n\log n)$ space.
\end{prop}

\begin{proof}
If $v$ is a geodesic for $u$, apply part one (Proposition \ref{prop:partone}) so that $v$ is of the form 
 $t^{-d} a^{\psi_0} t\ldots ta^{\psi_r}t^{-l}$ 
for some integers $d,r,l$, and $\psi_0\neq 0$ if $d>0$,  $|\psi_i|<p$ for $0\leq  i<r$ and $|\psi_r|<3p$. 
Replacing the $PN$ suffices of $u$ and $v$ by powers of $a$, then applying Britton's Lemma to $u^{-1}v$ proves that $d=0$.

So $v=a^{\psi_0} t\ldots ta^{\psi_r}t^{-l}$. Since $u,v$ have the same $t$-exponent sum we also have $q=r-l$, so we can write  $v=a^{\eta_0} t\ldots ta^{\eta_q}$  if $l=0$, or  $v=a^{\eta_0} t\ldots ta^{\eta_q}ta^{\rho_{q+1}}t\ldots ta^{\rho_{q+s}}t^{-s}$ with $s> 0$. 

This proves the first assertion. 

To compute $v$, repeat the procedure from Proposition \ref{prop:twotypeP}, tracking the path $u$ from the point $S$ at level 0 up to level $q$. So we have a line at level $q$ with points marked $U,d_1, (d_2)$. Relabel the points $ d_1, (d_2)$ as $d_{1,q},(d_{2,q})$, and the paths to these points $g_{1,q}, (q_{2,q})$.
A geodesic $v$ must travel from $S$ up to this level, so without loss of generality $v$ travels via the path $g_x (x=1,2)$, and then ends  with some  word equal to  $a^m$ for some power $m$. 
So $v=g_{x,q}v^*$ where $v^*$ is of the form $a^m$ or PN. Note that $|m|<3p$ otherwise $v$ is not geodesic. 

To compute $v^*$, we do the following. 

Draw the line at level $q$
with  $3p-1$ edges to the right and left of the point(s) $d_1,d_2$. See Figure \ref{fig:final2}.
 \begin{figure}[ht!]
\bc
\psfrag{U}{$U$}\psfrag{d1}{$d_{1,q}$}  \psfrag{d2}{$d_{2,q}$}   
\includegraphics[ width=11.8cm]{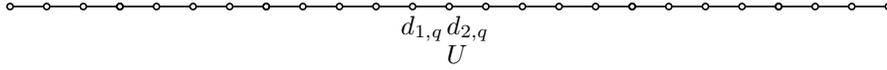} 
\ec\caption{Level $q$, with $3p-1=11$ edges on either side of $d_{1,q},d_{2,q}$ added.}\label{fig:final2}
\end{figure}
Read $a^{\epsilon_q}$ along the line from the point $U$, draw
 vertical $t$ edge up, and cover the line with bricks as before. Label the points on the next level up as $U,d_{1,q+1}, (d_{2,q+1})$. Do not discard the previous level as we did before. This time we will store all levels from $q$ to $q+r$.
 
 Again extend the line at level $q+1$ out by $3p-1$ edges on either side of these points, read $a^{\delta_{q+1}}$, and repeat. 
  \begin{figure}[ht!]
\bc
\psfrag{U}{$U$}\psfrag{d1}{$d_{1,q}$}  \psfrag{d2}{$d_{2,q}$} \psfrag{da}{$d_{1,q+1}$}  \psfrag{db}{$d_{2,q+1}$} 
\includegraphics[ width=11.8cm]{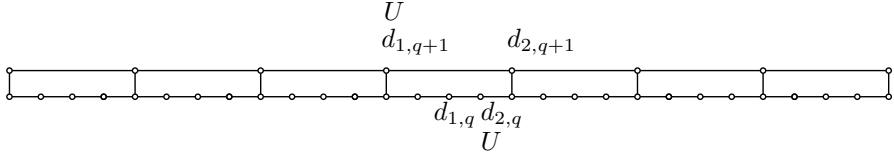} 
\ec\caption{Covering level $q$ with bricks. Here $\delta_{q+1}=1$.}
\end{figure}

Note that we are storing all these levels. Each level has a row of at most 7 bricks, so the amount of space required to store these levels with three labels on each level is $O(n\log n)$ since the labels take $O(\log n)$ space and the number of levels is $r\leq n$. Note also that we do not store each individual $g_{i,j}$ -- we store $g_{1,q+r},g_{2,q+r}$ only. 

At level $q+r$, from the point marked $U$, read $a^{\delta_{q+r}}$  and mark the endpoint by $E_r$. As before, if $v$ extends above this level, then it can go only a bounded number of levels more, so draw these layers of bricks in, aligned with the point $E_r$. So a geodesic to $E_r$ will be one of a finite number of paths, as before.
Choose a shortest path to $E_r$, append $t^{-r}$ to it, and store it as $v_r$. Note that $\ell(v_r)\leq n$ so we need $O(n)$ space to store the word.

Now the geodesic for $u$ could be $v_r$, or could be a path of the form $v_s=g_{x,q} a^{\eta_q}t\ldots ta^{\eta_q+s}t^{-s}$ with $s<r$, that is, a word that does not travel up as high as $u$, where Corollary \ref{cor:pusha} allows us to assume that all $a$ letters are pushed out of the $N$ suffix of $v$.

From the point marked $E_r$ in the stored diagram, draw a line of $t^{-1}$ edges as far down the diagram as possible, until either you reach level $q$, 
 or at some level the $t^{-1}$ edge goes out of  the diagram. If a path travels from the point labeled $U$ at level $q$ to the endpoint of the line $t^{-r}$ in the Cayley graph and travels more that $3p-1$ $a^{\pm 1}$ edges along some level, then it is not geodesic. This means that if a path for $u$ leaves the stored diagram, it is not geodesic. Therefore to locate our geodesic output  we simply must check all paths $v_s$ that lie inside the stored diagram.

Label the points along the line $t^{-r}$ from $E_r$ that stay within the diagram by $E_{r-1},\ldots, E_0$ where $E_i$ is at level $q+i$.

At level $q+r-1$, compute the length of the shortest path that travels via $g_{x,q+r-1}$ to $d_{x,q+r-1}$, the across to $E_{r-1}$ via $a^{\pm 1}$ edges, then ends by $t^{-r+1}$, where $x=1,2$. Call this path $v_{r-1}$. Compare the lengths of $v_r, v_{r-1}$ and store the shortest one. Note that the path(s)  $g_{x,q+r-1}$ are not stored, but are easily obtained by deleting their suffixes.

Repeat for each level below, storing the shortest path, and the word $g_{x,i}$ for all levels where $E_i$ is contained in the stored diagram.

When this process terminates, we have located a geodesic for the input word.
\end{proof}

\begin{prop}\label{prop:twotypeNPN}[$u$ of type $NPN$]
Let $u=t^{-k}a^{\epsilon_0}t\ldots ta^{\epsilon_q}ta^{\delta_{q+1}}t\ldots ta^{\delta_{q+r}}t^{-r}$ be the output of part one, where $q\geq k$, $r\geq 1$, $|\epsilon_i|<p$ for $0\leq i\leq q$, $|\delta_{q+i}|<p$ for $1\leq i<r$, $|\delta_{q+r}|<3p$, and $\ell(u)\leq n$, the length of the initial input word. 

Then    there is a geodesic $v$ for $u$ of the form $v=t^{-k}a^{\eta_0} t\ldots ta^{\eta_q}$  or  $v=t^{-k}a^{\eta_0} t\ldots ta^{\eta_q}ta^{\rho_{q+1}}t\ldots ta^{\rho_{q+s}}t^{-s}$ with $s> 0$,
which can be  computed in linear time and $O(n\log n)$ space.
\end{prop}

\begin{proof} Part one and Britton's lemma proves that any geodesic for $u$ is the form of part one must start with $t^{-k}$. We then repeat the procedure 
above with the start point $S$ at the endpoint of $t^{-k}$ rather than the identity, and the result follows.
\end{proof}

\bibliography{refs} \bibliographystyle{plain}

\end{document}